\documentclass[11pt,leqno,a4paper, english]{amsart}

\usepackage[]{hyperref}
\hypersetup{
    colorlinks=true,       
    linkcolor=red,          
    citecolor=blue,        
    filecolor=magenta,      
    urlcolor=cyan           
}

\usepackage{amsmath}
\usepackage{amsfonts,amssymb}
\usepackage{enumerate}
\usepackage{mathrsfs}
\usepackage{amsthm}
\usepackage{a4wide}
\usepackage[cp1251]{inputenc}

\theoremstyle{plain}
\newtheorem{theo}{Theorem}[section]
\newtheorem{prop}[theo]{Proposition}
\newtheorem{lemm}[theo]{Lemma}

\theoremstyle{definition}

\DeclareSymbolFont{pletters}{OT1}{cmr}{m}{sl}
\DeclareMathSymbol{s}{\mathalpha}{pletters}{`s}

\def\eps{\varepsilon}

\def\la{\left\lvert}

\def\mez{\frac{1}{2}}

\def\ra{\right\rvert}

\def\xC{\mathbf{C}}

\def\xR{\mathbf{R}}

\numberwithin{equation}{section}

\title  {A note on the uniqueness from sets of positive measure for time dependent parabolic operators}


\author[N. Burq]{Nicolas Burq}
\address{Universit{\'e} Paris-Saclay, Math{\'e}matiques, UMR 8628 du CNRS, B{\^a}t 307, 91405  Orsay Cedex, France,   and Institut Universitaire de France}
\email{Nicolas.burq@universite-paris-saclay.fr}
\author[C. Zuily]{Claude Zuily}
\address{Universit{\'e} Paris-Saclay, Math{\'e}matiques, UMR 8628 du CNRS, B{\^a}t 307, 91405  Orsay Cedex, France, }
\email{Claude.zuily@universite-paris-saclay.fr}
\begin{document}
\maketitle
\begin{abstract} The purpose of this short note is to show how it is possible to combine existing results in the literature to get the unique continuation from sets of positive measure for time dependent parabolic equations with Lipschitz principal part and bounded lower order terms, result which was known in the case of analytic coefficients in \cite{EMZ}.
\end{abstract}

\section{Introduction}
The problem of unique continuation for solutions to parabolic equations has received a lot of attention starting from the pionneering work of T. Carleman in 1939. Since that time a huge number of papers have been devoted to this question.   The interested reader may consult the review paper   \cite{V} and its bibliography. Most of these works deal with the question of weak continuation from   open sets \cite{SS}, strong continuation from a point \cite{KT}, quantitative estimates etc..

Recently, motivated by   control theory, the question of unique continuation from sets of positive measure has been discussed and   several new results for elliptic equations (see \cite{LM}),   and for parabolic equations with  analytic coefficients have been proved (see \cite{EMZ}).

The purpose of this note is to handle the case of time dependent equations with Lipschitz coefficients and to show that known results in the literature can be recombined to give a positive answer to this question.

\subsection{Notations}

For $T>0, R>0$  let $ I= (-T,T) $ and $B_R = \{x\in \xR^d: \vert x \vert<R\}.$  We set $Q = I \times B_R.$

We shall consider the parabolic operator,
\begin{equation}    \label{operateur}
 P= \partial_t -\text{div}\,(A(t,x)\nabla_x) + \sum_{j=1}^d b_j(t,x)\partial_{x_j} + c(t,x),
 \end{equation}
where $A(t,x) = (a_{jk}(t,x))_{1 \leq j,k \leq d}$ is a symmetric matrix with real valued entries satisfying the ellipticity condition,
\begin{equation}\label{ellipt}
\exists \kappa >0: \text{Re}\, \langle A(t,x)  \zeta,  \zeta\rangle \geq \kappa \vert \zeta\vert^2, \quad \forall (t,x) \in Q, \quad \forall \zeta \in \xC^d,
\end{equation}
  the   $(a_{jk})$ are uniformly Lipschitz continuous in $\overline{I }\times \overline{B_R}$ and $b_j, c$ belong  to  $L^\infty(\overline{I }\times \overline{B_R}).$

Let $(t_0, x_0)\in (-T,T)\times B_R.$  We shall set, for small $r>0$,
\begin{align*}
 I_r(t_0)  = \{t\in \xR: \vert t-t_0\vert < r^2\}, &\quad B_r(x_0) = \{x\in \xR^d: \vert x-x_0\vert < r\}, \\
 \quad Q_r(t_0, x_0)&= I_r(t_0)\times B_r(x_0). 
 \end{align*}

Eventually, we shall denote by $\mu$ the Lebesgue measure on $\xR^d.$
\subsection{Main result}

The purpose of this note is to prove the following result.

\begin{theo}\label{unicite}
Let $u \in L ^2((-T,T), H ^2(B_R))$ satisfying $  Pu =0$ in $ Q.$ Assume that there exists $E \subset B_R$ with $\mu(E)>0$ such that $u(t,x) =0$ for all $(t,x)\in (-T,T) \times E.$  
Then  $u=0$ on $Q.$
\end{theo}
\subsection{Preliminaries}
This result will be a consequence of several  results that we recall now. The first one is the well known  Cacciopoli inequality.  
\begin{prop} \label{Caccio}
  Let $u \in L ^2((-T,T), H ^2(B_R))$ satisfying $  Pu  =0  $ in $Q.$ One can find a constant $C>0$ depending only on $d$ and the ellipticity constant $\kappa$    such that for every $r >0$ such that $Q_{2r}(x_0,t_0) \subset Q$ we have,
$$\iint_{ Q_r(t_0, x_0)} \vert \nabla_x u(t,x)\vert^2\, dx\, dt \leq \frac{C}{r^2}\iint_{ Q_{2r}(t_0, x_0)} \vert u(t,x)\vert^2\, dx\, dt$$
\end{prop} 
\begin{proof}
Without loss of generality one may assume that $(t_0,x_0)= (0,0).$ Let $\theta_1 \in C_0^\infty(\xR), \theta_1(t)= 1$ for $\vert t \vert \leq 1, \text{supp}\, \theta_1 \subset \{\vert t \vert \leq 2\}$ and $\theta_2 \in C_0^\infty(\xR^d), \theta_2(x)= 1$ for $\vert x \vert \leq 1, \text{supp}\, \theta_2 \subset \{\vert x \vert \leq 2\},$ $0 \leq \theta_j \leq 1.$ We set for small $r$, $\chi(t,x) = \theta_1(\frac{t}{r^2}), \theta_2(\frac{x}{r}).$   Then, 
\begin{equation}\label{caccio1}
  \vert \partial_t \chi(X)\vert + \vert \nabla_x \chi (X)\vert^2 \leq \frac{C}{r^2}, \quad X =(t,x).
  \end{equation}
Denoting by $((\cdot, \cdot))$ the scalar product in $L^2((-T,T)\times \xR^d)$ we have,
\begin{equation}\label{caccio2}
 ((Pu, \chi^2u))= ((F,\chi^2u)), \quad  \vert F \vert \leq C(\vert u \vert + \vert \nabla_x u\vert). 
\end{equation}
 
Now   and using the fact that $\chi$ has compact support in $(-T,T)\times B_R$ we have,
$$  \text{Re} \, ((\partial_t u, \chi^2 u))= \mez \iint  \chi^2(X) \frac{d}{dt}\vert u(X) \vert^2  \, dX = -   \iint  \chi(X) (\partial_t \chi)(X)  \vert u(X) \vert^2 \, dX, $$
and, 
 \begin{align*}
   -   \text{Re} \,((\text{div}\, (&A\nabla_x  u),    \chi^2 u ))\\
   &=   \text{Re} \iint A(X)\nabla_x u(X) \Big(\chi^2(X) \nabla_x u(X) + 2 \chi(X) \nabla_x \chi(X) u(X)\Big)\, dX. 
\end{align*}
Using \eqref{caccio2}, the hypothesis \eqref{ellipt} and the Cauchy-Schwarz inequality we obtain,
\begin{align*}
 &\kappa \Vert \chi\vert \nabla_x u\vert\Vert^2 \leq \frac{C}{r^2} \iint \chi(X)\vert u(X)\vert^2\, dX  + 2\Vert A\Vert_{L^\infty(Q)} \Big(\iint \vert \chi(X)\vert^2\vert \nabla_x u(X)\vert^2\, dX\Big)^\mez\\
  &  \cdot\Big(\iint \vert \nabla_x \chi(X)\vert^2 \vert u(X)\vert^2\, dX\big)^\mez + C\iint \chi^2(X) \vert u(X)\vert^2\, dX + C  \Big(\iint \vert \chi(X)\vert^2\vert \nabla_x u(X)\vert^2\, dX\Big)^\mez\\
  &\cdot \Big(\iint \vert \chi(X)\vert^2\vert   u(X)\vert^2\, dX\Big)^\mez .
  \end{align*}
Using the inequality $ab \leq \eps a^2 + \frac{1}{4\eps} b^2$, the estimates \eqref{caccio1}, and the fact that $r$ is small we obtain,
$$\mez \kappa \iint \vert \chi(X)\vert^2 \nabla_x u(X)\vert^2\, dX  \leq \frac{C'}{r^2} \iint _{I_{2r}\times Q_{2r}} \vert u(X)\vert^2\, dX.$$
The conclusion follows from the fact that $\chi(X) = 1$ on $I_r\times Q_r.$

\end{proof}
Notice that this Proposition appears in the work by P.Auscher, S.Bortz, M.Egert, O.Saari (\cite{ABES} Proposition 4.3).

The second one is a result   by L.Escauriaza, F.J.Fernandez, S.Vessela (\cite{EFV} Theorem 3).
\begin{theo}(\cite{EFV}) \label{EFV}
 Let $P$ be defined in \eqref{operateur} and let $u\in  L ^2((-T,T), H ^2(B_R))$ be a solution of $Pu=0$ in $Q.$  Set,  
$$ \Theta= \frac{\iint_{Q_4}\vert u(t,x)\vert^2\, dx\, dt}{\int_{B_1(x_0)} \vert u(0,x)\vert^2\, dx}.$$ 
Then there exists $N= N(d, \kappa)>0$ such that the following holds when $0<r< (N \text{Log}(N \Theta))^{-\mez},$
$$\iint_{Q_{2r}(t_0, x_0)} \vert u(t,x)\vert^2\, dx \, dt \leq D(N,u)\iint_{Q_{r}(t_0, x_0)} \vert u(t,x)\vert^2\, dx \, dt$$
where $D(N,u)= \text{exp}\big( N \text{Log}(N\Theta) \text{Log}( N\text{Log}(N\Theta)\big).$
\end{theo}

The third result  is  the following lemma,  a kind of Poincar\'e inequality,  which  can be found in  the book by O. Ladyzenskaya, N.Uraltseva \cite{LU}.

\begin{lemm} (\cite{LU}) \label{lad}
Let   $x_0\in \xR^d$ and $r>0.$, Let $F$ be a measurable subset of $B_r(x_0)$ with $\mu(F)>0.$ Then  for every   $v \in H^1(B_r(x_0))$   and for every measurable set $A \subset B_r(x_0)$ we have,
$$\mu(F)\int_A\vert v(x)\vert^2\, dx \leq  2\mu(A)\int_F\vert v(x)\vert^2\, dx+ C\, r^{d+1}\,  \mu(A)^{\frac{1}{d}} \int_{B_r(x_0)}\vert \nabla_x v(x)\vert^2\, dx,$$
where $C =  \frac{ 2 ^{d+3}}{d+1} \big(\mu(S^{d-1})\big)^{1- \frac{1}{d}}.$
\end{lemm}
For the reader's convenience we give the proof of this lemma.
\begin{proof}[Proof of Lemma \ref{lad}]
It is of course enough to prove this lemma for $x_0 = 0$ and for $v\in C^\infty(\overline{B_r})$  where $B_r = B_r(0).$ Let $y\in F$ and $x \in A, x \neq y.$  We have,
\begin{align*}
 v(y)-v(x)  & =  \int_{0}^{\vert y-x\vert} \frac{d}{d\rho}\Big(v(x+ \rho \frac{y-x}{\vert y-x\vert}\Big)\, d\rho,\\
 &= \int_{0}^{\vert y-x\vert}\nabla_x v\Big(x+ \rho \frac{y-x}{\vert y-x\vert}\Big)\cdot \frac{y-x}{\vert y-x\vert}\, d\rho.
 \end{align*}
 Using the Cauchy-Schwarz inequality we obtain,
 $$\vert v(x)\vert^2 \leq 2 \vert v(y)\vert^2+ 2\vert y-x\vert \int_{0}^{\vert y-x\vert} \la \nabla_x v\Big(x+ \rho \frac{y-x}{\vert y-x\vert}\Big)\ra^2\, d\rho.$$
 Integrating this inequality  for  $y\in F$ we get,
 \begin{equation}\label{ineg1}
  \mu(F) \vert v(x)\vert^2 \leq 2 \int_F\vert v(y)\vert^2\, dy+  \int_{B_r}\vert y-x\vert \int_{0}^{\vert y-x\vert} \la \nabla_x v\Big(x+ \rho \frac{y-x}{\vert y-x\vert}\Big)\ra^2\, d\rho\, dy.
  \end{equation}
 We would like to set, in the second integral in the right hand side,  $y= x+t\omega, t>0, \omega\in S^{d-1}.$ We first prove that, for $x \in B_r$,
 \begin{equation}\label{equiv}
  y= x+t\omega \in B_r,  t>0, \omega\in S^{d-1}\Longleftrightarrow y= x+t\omega, 0<t<t^*(x, \omega),
  \end{equation}
 with $0<t^*\leq 2r.$
 
 Indeed $y \in B_r$ is equivalent  to $\vert x+t\omega\vert^2 < r^2$ thus to $t^2 + 2(x\cdot\omega)t + \vert x\vert^2-r^2<0.$ The reduced discriminant of this second order polynomial in $t$ is equal to $\Delta = (x \cdot \omega)^2+ r^2-\vert x \vert^2$ which is strictly positive. Therefore this polynomial has two real roots  and since their product is $\vert x \vert^2-r^2<0,$ there are one positive root  $t^*$ and one negative root $t_*$. Therefore this polynomial is strictly negative if and only if $0<t<t^*.$  Moreover $t^*= - x \cdot \omega + \sqrt{ (x \cdot \omega)^2+ r^2-\vert x \vert^2}$ therefore $0<t^* \leq \vert x \vert + r \leq 2r.$
 
 Making the change of variable $y= x+t\omega$ in the second integral in the right hand side of \eqref{ineg1} we obtain,
 $$ \mu(F) \vert v(x)\vert^2 \leq 2 \int_F\vert v(y)\vert^2\, dy + 2\int_{S^{d-1}} \int_{t=0}^{t^*}t \int_{\rho = 0}^t \vert \nabla_x v(x+ \rho \omega)\vert^2 t^{d-1}\, d\rho \,  dt\, d \omega.$$
 By the Fubini Theorem one can write,
 $$ \mu(F) \vert v(x)\vert^2 \leq 2 \int_F\vert v(y)\vert^2\, dy +  \int_{S^{d-1}}\int_{\rho=0}^{t^*} \Big( \int_{t= \rho}^{t^*} t^d \, dt\Big)\vert \nabla_x v(x+ \rho \omega)\vert^2 \, d\rho\, d\omega.$$
 Since $t^* < 2r$ we obtain,
 $$ \mu(F) \vert v(x)\vert^2 \leq 2 \int_F\vert v(y)\vert^2\, dy +  \frac{ 2 ^{d+2}}{d+1} r^{d+1} \int_{S^{d-1}}\int_{\rho=0}^{t^*}  \frac{\vert \nabla_x v(x+ \rho \omega)\vert^2}{\rho^{d-1} } \, \rho^{d-1} d\rho\, d\omega.$$
Setting $z = x + t \omega$ and using again \eqref{equiv} we obtain,
$$\mu(F) \vert v(x)\vert^2 \leq 2 \int_F\vert v(y)\vert^2\, dy + \frac{ 2 ^{d+2}}{d+1} r^{d+1} \int_{B_r} \frac{\vert \nabla_x v(z)\vert^2}{\vert z-x\vert^{d-1}}\, dz.$$
 Integrating both members of this inequality with respect to $x$ on $A$ we get,
 \begin{equation}\label{ineg2}
  \mu(F) \int_A \vert v(x)\vert^2\, dx \leq  2 \mu(A)\int_F\vert v(y)\vert^2\, dy +\frac{ 2 ^{d+2}}{d+1} r^{d+1}\int_{B_r} \vert \nabla_x v(z)\vert^2 g(z)\, dz,
  \end{equation}
 where $g(z) = \int_{A}  \frac{1}{\vert z-x\vert^{d-1}}\, dx.$ 
 
 To compute $g$ let  $\delta >0$ and set $A_1= \{x \in A: \vert z-x\vert < \delta\}, A_2= \{x \in A: \vert z-x\vert > \delta\}.$ We have,
 \begin{align*}
 &\int_{A_1}  \frac{1}{\vert z-x \vert^{d-1}}\, dx \leq \int_{\vert z-x\vert<\delta}  \frac{1}{\vert z-x \vert^{d-1}} \, dx \leq \mu(S^{d-1}) \, \delta,\\
 &\int_{A_2}  \frac{1}{\vert z-x\vert^{d-1}}\, dx \leq \delta^{-(d-1)} \mu(A).
 \end{align*}
We then chose $\delta$ such that $\mu(S^{d-1}) \, \delta=\delta^{-(d-1)} \mu(A),$ that is, $ \delta = \Big(\frac{\mu(A)}{\mu(S^{d-1}) }\Big)^{\frac{1}{d}}$ and we obtain the estimate,
$$g(z) \leq 2 \big(\mu(S^{d-1})\big)^{1- \frac{1}{d}} \mu(A)^{\frac{1}{d}}.$$
 Using \eqref{ineg2} we obtain eventually,
 $$ \mu(F) \int_A \vert v(x)\vert^2\, dx \leq   2\mu(A) \int_F\vert v(y)\vert^2\, dy+ \frac{ 2 ^{d+3}}{d+1} \big(\mu(S^{d-1})\big)^{1- \frac{1}{d}}r^{d+1} \mu(A)^{\frac{1}{d}}\int_{B_r} \vert \nabla_x v(z)\vert^2\, dz.$$

\end{proof}
\subsection{Proof of Theorem \ref{unicite}}
We shall use the method described in R.Regbaoui   \cite{R}.

We know that almost every point of $E$ is of  density     $1.$  If $x_0$ is such a point this means that,
$$\lim_{r \to 0} \frac{\mu(E \cap B_r(x_0))}{\mu(B_r(x_0))} = 1.$$
  Let $\eps \in (0,\mez)$ then there exists $r_0>0$ such that for $r\leq r_0$,
\begin{equation}\label{densite}
 \mu(E\cap B_r(x_0)) \geq (1-\eps) \mu(B_r(x_0)) , \quad \mu(E^c\cap B_r(x_0)) \leq  \eps \mu(B_r(x_0)).\end{equation}
We apply Lemma \ref{lad} to $v(x)= u(t,x), F = E\cap B_r(x_0), A= E^c\cap B_r(x_0).$ Since $u(t,x)=0$ for $x\in E$, we obtain,
$$
 \int_{B_r(x_0)} \vert u(t,x)\vert^2\, dx  = \int_{E^c \cap B_r(x_0)} \vert u(t,x)\vert^2\, dx 
   \leq   C_d r^{d+1} \frac{\mu(E^c \cap B_r(x_0))^{\frac{1}{d}}}{\mu(E \cap B_r(x_0))}\int_{B_r(x_0)} \vert \nabla_x u(t,x)\vert^2\, dx.
$$
Using \eqref{densite} we obtain, 
\begin{align*}
   \int_{B_r(x_0)} \vert u(t,x)\vert^2\, dx  &\leq 2C_d \,r^{d+1} \eps^{\frac{1}{d}}   \mu(B_r(x_0))^{1-\frac{1}{d}}\int_{B_r(x_0)} \vert \nabla_x u(t,x)\vert^2\, dx,\\ 
    &\leq C'_d \, r^2  \eps^{\frac{1}{d}}   \int_{B_r(x_0)} \vert \nabla_x u(t,x)\vert^2\, dx,
    \end{align*}

  since $\mu(B_r(x_0))= c_d r^d.$ 
  
  Integrating both members of this  inequality with respect  to $t$ on $I_r(t_0)= (t_0-r^2, t_0+ r^2)$ where $t_0$ is an arbitrary point of $I,$   we obtain, 
$$ \iint_{Q_r(t_0,x_0)} \vert u(t,x)\vert^2 \, dx\, dt \leq C'_d  \, \eps^{\frac{1}{d}}  r^2  \iint_{Q_r(t_0,x_0)} \vert \nabla_x u(t,x)\vert^2 \, dx\, dt.$$
Using Theorem \ref{Caccio} we deduce that, for $r \leq r_0,$
\begin{equation}\label{Ausch}
  \iint_{Q_r(t_0,x_0)} \vert u(t,x)\vert^2 \, dx\, dt \leq C''_d   \,\eps^{\frac{1}{d}} \iint_{Q_{2r}(t_0,x_0)} \vert  u(t,x)\vert^2 \, dx\, dt.
  \end{equation}
 
By Theorem \ref{EFV}   there exists $N= N(d,\kappa)$ such that when $0<r< (N\text{Log}(N \Theta))^{-\mez}$ we have,
$$\iint_{Q_{2r} (t_0, x_0)} \vert u(t,x)\vert^2\, dx\, dt \leq D(N,u)\iint_{Q_{r} (t_0, x_0)} \vert u(t,x)\vert^2\, dx\, dt,$$
where $D(N,u) = e^{N\text{Log}(N \Theta)\text{Log}(N\text{Log}(N \Theta))}.$ 

Combining this inequality with \eqref{Ausch}, we obtain,
$$ \iint_{Q_r(t_0,x_0)} \vert u(t,x)\vert^2 \, dx\, dt \leq C''_d   \,\eps^{\frac{1}{d}} D(N,u)\iint_{Q_{r} (t_0, x_0)} \vert u(t,x)\vert^2\, dx\, dt.$$
Taking $\eps \in (0, \mez)$ so small that, $C''_d  \eps^{\frac{1}{d}} D(N,u)<1$ we deduce from the above inequality that there exists $r>0$ such that $ u= 0$ in $Q_{r} (t_0, x_0).$  This argument holds for every $t_0 \in (-T,T).$ Therefore applying Theorem 1.1 in \cite{SS} or Theorem 1 in \cite{KT} we deduce that $u=0$ in $Q.$

\end{document}